\documentclass[a4paper,12pt]{article}
\usepackage{cite,amsmath,amssymb,amsthm}
\usepackage{circuitikz}
\usetikzlibrary{shapes.geometric}
\usepackage{caption}
\usepackage{subcaption}
\usepackage{hyperref}
\usepackage{float}
\usepackage[utf8]{inputenc}
\newtheorem{theorem}{Theorem}
\newtheorem{proposition}[theorem]{Proposition}

\newtheorem{corollary}[theorem]{Corollary}

\newtheorem{definition}[theorem]{Definition}
\newtheorem{problem}[theorem]{Problem}

\newtheorem{remark}[theorem]{Remark}

\tikzset{
  vertex/.style={circle, draw, minimum size=1.6mm, inner sep=0pt, fill=black, text=white},
  edge/.style={thick, black},
  removed/.style={thick, dashed, red}
}
\begin{document}
\tikzset{My Style/.style={draw, circle, fill=black,scale=0.3}}

\title{Šolt\'{e}s problem for the Kirchhoff index of a graph}

\author
{Kurt Klement Gottwald\thanks{Chemnitz University of Technology,
Faculty of Mathematics, Chemnitz, Germany. E-Mail: \texttt{kurt-klement.gottwald@mathematik.tu-chemnitz.de}},
\quad Snježana Majstorović Ergotić\thanks{Josip Juraj Strossmayer University of Osijek, School of Applied Mathematics and Informatics, Osijek, Croatia. E-Mail:
\texttt{smajstor@mathos.hr}}\\
\quad Tomislav Došlić\thanks{ University of Zagreb,  Faculty of Civil Engineering, Zagreb, Croatia.  E-Mail:
\texttt{tomislav.doslic@grad.unizg.hr}}}

\maketitle

\begin{abstract}
The Kirchhoff index  $Kf(G)$ of a connected graph $G$ is defined as
the sum of resistance distances between all pairs of vertices in $G$. We say that $v\in V(G)$ is a \textit{good vertex} if the Kirchhoff index remains unchanged when $v$ is removed, i.e. $Kf(G)=Kf(G-v)$.
In 1991, Šolt\'{e}s studied the Wiener index of a graph and posed the problem of identifying graphs for which the removal of an arbitrary vertex preserves the Wiener index.  In this paper, we explore a similar concept: identifying \textit{Kirchhoff Šolt\'{e}s graphs}, i.e. graphs in which all vertices are good vertices. We show that the cycle $C_5$ is a Kirchhoff Šolt\'{e}s graph. Due to the challenge of finding more examples of such graphs, we shift our focus to several relaxed versions of the Kirchhoff Šolt\'{e}s problem, where the primary objective is to identify graphs containing at least one good vertex.
One of them is the \textit{$\beta$-Kirchhoff Šolt\'{e}s problem}, which seeks to find an infinite family of graphs in which the proportion of good vertices is at least $\beta$, with $\beta \in (0,1]$ being a specified rational number. Another one involves constructing infinite families of graphs where the proportion of good vertices increases and asymptotically approaches a given real number $\gamma\in (0,1]$  as the order of the graph grows. 
We demonstrate that both relaxed versions have infinitely many solutions. In particular, we prove the existence of infinitely many graphs for which the proportion $\beta$ of good vertices,  $1/7\leq \beta<1/5$ tends to a certain irrational number.  Furthermore, we prove the existence of infinitely many graphs with half good vertices, and for each $s\in\mathbb{N}$, we construct an infinite family of graphs whose proportion of good vertices tends to $\frac{s+1}{2s+1}$.  These findings could be pivotal in addressing the original problem of determining whether there are additional solutions beyond $C_5$. 
\end{abstract}

\noindent\textbf{Keywords:}
Kirchhoff Šolt\'{e}s problem, Kirchhoff index, resistance distance, resistance transmission, Laplacian matrix, good vertex


\section{Introduction}

Effective resistance is a key concept in electric circuit theory that has been extensively explored by prominent figures in physics and engineering such as Kirchhoff \cite{kirch}, Maxwell \cite{max}, Seshu \& Reed \cite{reed}, and Chan \cite{chan}. 
While the graph-theoretic study of electric circuits began with Kirchhoff's analysis over 170 years ago, the introduction of effective resistance as a novel graphical distance occurred in 1993 by Klein \& Randić \cite{klein1}. 
Prior to that, the "shortest-path" distance was the primary type of graphical distance widely recognized and extensively studied. Thus, the introduction of effective resistance as a new form of graphical distance generated significant interest. Scientists eventually recognized resistance distance as a fundamental concept in graph theory, leading to explorations of its mathematical properties and applications in fields such as chemistry.\\
\indent To explain the notion of resistance distance, 
we associate a graph $G$ with a network $N(G)$ of unit resistors (resistor with resistance $1$). In this network, each edge of $G$ corresponds to a unit resistor. See Figure \ref{fig1} for an illustration.

\begin{center}
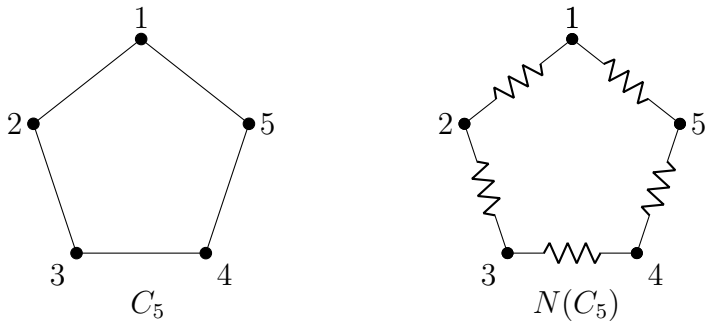
\begin{figure}[!htb]
    \centering
\begin{tikzpicture}[scale=0.95]
\draw[fill=black] (1.2,0) circle (0.8mm);
\draw[fill=black] (-0.3,-1.18) circle (0.8mm);
\draw[fill=black] (0.3,-2.98) circle (0.8mm);
\draw[fill=black] (2.1,-2.98) circle (0.8mm);
\draw[fill=black] (2.7,-1.18) circle (0.8mm);
\node[anchor=south] at (1.2,0) {1};
\node[anchor=east] at (-0.3,-1.18) {2};
\node[anchor= north east] at (0.3,-2.98)  {3};
\node[anchor= north west] at (2.1,-2.98) {4};
\node[anchor= west] at  (2.7,-1.18) {5};
\draw (1.2,0)--(-0.3,-1.18)--(0.3,-2.98)--(2.1,-2.98)--(2.7,-1.18)--(1.2,0);
\node[text width=6cm, anchor=north, right] at (0.9,-3.7)
    {$C_5$};
\ctikzset{bipoles/length=0.9cm, nodes width=0.08}
\draw (7.2,0) node [anchor=south] {1} to [R,*-*](5.7,-1.18) node [anchor=east] {2} to [R,*-*] (6.3,-2.98) node [anchor= north east] {3} to [R,*-*] (8.1,-2.98) node [anchor= north west] {4} to [R,*-*] (8.7,-1.18) node [anchor= west] {5} to [R,*-*] (7.2,0) node [anchor= south] {1};
\node[text width=6cm, anchor=north, right] at (6.5,-3.7)
    {$N(C_5)$};
\end{tikzpicture}
\caption{The cycle $C_5$ and the associated network $N(C_5)$ of unit resistors.}
\label{fig1}
\end{figure}
\end{center}

\begin{center}
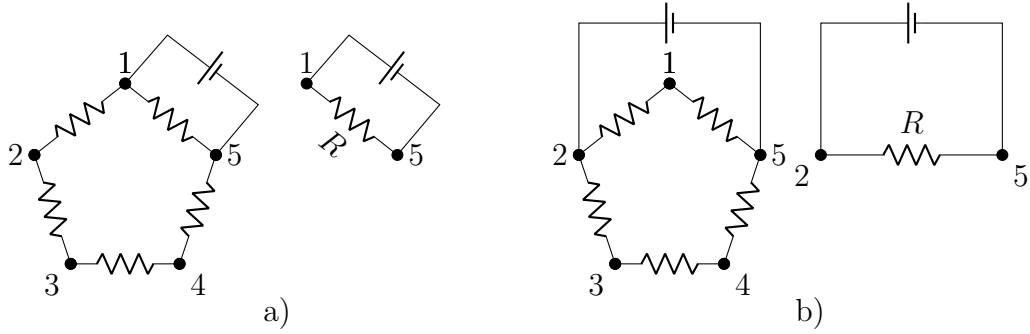
\begin{figure}[!htb]
 \centering
\begin{tikzpicture}[scale=0.8]
\ctikzset{bipoles/length=0.9cm, nodes width=0.08}
\ctikzset{label/align = smart}
\draw (1.2,0) node [anchor=south] {1} to [R,*-*]
 (-0.3,-1.18) node [anchor=east] {2} to [R,*-*] 
(0.3,-2.98) node [anchor= north east] {3} to [R,*-*]
(2.1,-2.98) node [anchor= north west] {4} to [R,*-*]
(2.7,-1.18)node [anchor= west] {5} to [R,*-*] 
(1.2,0)  node [anchor= south] {1};
\draw (1.2,0)--(1.9, 0.8) to[battery1] (3.4,-0.35)--(2.7,-1.18);
\ctikzset{bipoles/length=0.9cm, nodes width=0.08}
\draw (4.2,0) node [anchor=south] {1} to [R, l_=$R$,*-*]
(5.7,-1.18)node [anchor= west] {5};
\draw (4.2,0)--(4.9, 0.8) to[battery1] (6.4,-0.35)--(5.7,-1.18);
\node[text width=6cm, anchor=north, right] at (3.3,-3.8)
    {a)};
\ctikzset{bipoles/length=0.9cm, nodes width=0.08}
\draw (10.2,0) node [anchor=south] {1} to [R,*-*](8.7,-1.18) node [anchor=east] {2} to [R,*-*] (9.3,-2.98) node [anchor= north east] {3} to [R,*-*] (11.1,-2.98) node [anchor= north west] {4} to [R,*-*] (11.7,-1.18) node [anchor= west] {5} to [R,*-*] (10.2,0) node [anchor= south] {1};
\draw (8.7,-1.18)--(8.7, 1) to[battery1] (11.7,1)--(11.7,-1.18);
\node[text width=6cm, anchor=north, right] at (12.1,-3.8)
    {b)};
   \ctikzset{bipoles/length=0.9cm, nodes width=0.08}
\draw (12.7,-1.18) node [anchor= north east] {2} to [R, l=$R$,*-*]
(15.7,-1.18)node [anchor=north west] {5};
\draw (12.7,-1.18)--(12.7, 1) to[battery1] (15.7,1)--(15.7,-1.18);
\end{tikzpicture}
\caption{An electrical circuit obtained from $N(C_5)$ with  effective resistance $R$ between nodes a) $1$ and $5$ and b) $2$ and $5$.}
\label{fig2}
\end{figure}
\end{center}
If a source of electromotive force is connected to two nodes in the network, current will flow into and out of the network. According to Ohm’s law, if the potential difference between two nodes (vertices) $a$ and $b$ in $N(G)$ is $V$ and the current flowing into one node and out of the other is $I$, then $V = IR$, where $R$ is the effective resistance between the two nodes. See the example on Figure \ref{fig2}. 
Ohm's law provides straightforward formulas for calculating the effective resistance of resistors connected in series or in parallel.
If $k$ resistors with resistances of $r_1, r_2,\ldots, r_k$ ohms are connected in series, their effective resistance is
\begin{equation}\label{series}
\Omega(a, b) = r_1 + r_2 + \cdots + r_k,
\end{equation} 
while if they are in parallel, their effective resistance $\Omega(a, b)$ satisfies
\begin{equation}\label{paral}
\frac{1}{\Omega(a, b)} = \frac{1}{r_1} + \frac{1}{r_2} + \cdots + \frac{1}{r_k}.
\end{equation}
By using formulas (\ref{series}) and (\ref{paral}) on the network $N(C_5)$ with a battery linking vertices $1$ and $5$ (see Figure \ref{fig2}a) and $2$ and $5$ (see Figure \ref{fig2}b), we obtain the respective resistance distances
\begin{align}
      R=\Omega(1,5)&=1/\left(1+\frac{1}{4}\right)=\frac{4}{5},\\
      R=\Omega(2,5)&=1/\left(\frac{1}{2}+\frac{1}{3}\right)=\frac{6}{5}. 
\end{align}

\indent Throughout this paper we consider a graph $G$ to be finite, simple, undirected, and connected. The vertex set of $G$ is denoted by $V(G)$, and the edge set by  $E(G)$. The number of vertices in $G$ is usually denoted by $n(G)$ or simply $n$, while the number of edges is denoted by $m(G)$ or $m$. The distance $d_G(u, v)$ (or simply $d(u,v)$) between vertices $u$ and $v$ (which we previously referred to as the "shortest-path" distance) is defined
as the number of edges in the shortest path connecting them in $G$.

\begin{definition} The resistance distance between vertices $a$ and $b$ in  a graph $G$, denoted by $\Omega(a,b)$ is the effective
resistance between $a$ and $b$ in the network of unit resistors $N(G)$.
\end{definition}
We use the notation $\Omega_G (a,b)$  to emphasize the graph $G$ in which we consider the resistance distance between vertices $a$ and $b$.
Klein \& Randić \cite{klein1}  proved that  the function $\Omega: V(G) \times V(G) \to \mathbb{R}$  is a metric on the set $V(G)$.
Analogous to the definition of the Wiener index $W(G)$, which is the sum of the shortest-path distances between all unordered pairs of vertices in  $G$ \cite{wiener}, 
the Kirchhoff index $Kf(G)$ is defined as the sum of the resistance distances between all unordered pairs of vertices in $G$, i.e.
\begin{equation}\label{kirch}
Kf(G)=\sum_{\{u,v\}\subseteq V(G)}\Omega(u, v) = \frac{1}{2}\sum_{v\in V(G)}Rt(v),
\end{equation}
where $Rt(v)=\displaystyle{\sum_{u\in V(G)}\Omega (u,v)}$ is the resistance transmission of the vertex $v$.  Klein \& Randić proved that $Kf(G) \leq W(G)$ with equality holding if and only if $G$ is a tree.
Since the Wiener index, the oldest and most studied distance-based topological index, has been deeply investigated in the context of trees, the Kirchhoff index is primarily of interest for cyclic graphs.\\
\indent In recent years, the Kirchhoff index has been the focus of intense scrutiny from various perspectives. 
Research on the Kirchhoff index typically addresses three main areas: determining exact values for the Kirchhoff index in graphs endowed with some form of symmetry or special property, establishing general bounds for the Kirchhoff index in terms of graph invariants, and identifying extremal graphs within specific families. Below, we highlight several significant contributions related to resistance distance and the Kirchhoff index.
Fowler \cite{fow} calculated resistance distances in fullerene graphs, Zhang and Jang \cite{zhang} derived closed-form formulae for the Kirchhoff index and resistance distances of circulant graphs,  Bapat \& Gupta \cite{bapat2} presented 
formulas involving Fibonacci numbers for resistance distances in wheels and fans, while Gervacio \cite{gerva} provided explicit expression for resistance distances between any pair of vertices in complete multipartite graphs. The concept of resistance distance was extended to directed graphs using random walks by researchers in \cite{bianchi, zhu}. Majorization theory was used to obtain upper and lower bounds for the Kirchhoff index in arbitrary graphs \cite{bianchi2}, and extremal unicyclic graphs with respect to the Kirchhoff index were identified \cite{unicyc}. Additionally, the behavior of resistance distance under various unary and binary graph operations was examined \cite{yang, zhang2}.\\
\indent The present paper is motivated by the {\v S}olt{\'e}s problem   on the Wiener index posed in 1991:
\begin{problem} \cite{soltes}\label{prob1} Find all graphs $G$ so that the equality $W(G) = W(G-v)$ holds for all $v\in V(G)$. We know just one such graph - the cycle on $11$ vertices. 
\end{problem}
The problem remained untouched for almost 30 years until it was revived by Knor et al. in 2018 \cite{knor}. Since then, several relaxed versions of the problem have been successfully addressed  \cite{akhm,bok,hu}. However, the original version of the problem remains unsolved. \\
\indent In our study, we introduce a new variant of the {\v S}olt{\'e}s  problem as follows.
\begin{problem}\label{prob2} Find all graphs $G$ so that the equality $Kf(G) = Kf(G-v)$ holds for all $v\in V(G)$.
\end{problem}
To distinguish it from the original problem, we will refer to Problem \ref{prob2}  as the Kirchhoff \v{S}olt\'{e}s problem. Unlike some graph invariants that consistently increase or decrease when any vertex is removed from any graph, the Kirchhoff index can exhibit any of three possible outcomes, depending on the choice of the graph and/or the vertex.

\noindent \textbf{Outcome 1.} $Kf(G)<Kf(G-v)$, i.e. the Kirchhoff index increases. In \cite{luk} it was proved that for an $n-$vertex cycle graph $C_n$ it holds $Kf(C_n)=\frac{n^3-n}{12}$. If we take $n=11$, then $Kf(C_{11})=110$ and $Kf(C_{11}-v)=Kf(P_{10})$ and from \cite{klein1} we know $Kf(P_{10})=W(P_{10})=165$.\\
\noindent \textbf{Outcome 2.} $Kf(G)>Kf(G-v)$, i.e. the Kirchhoff index decreases. For a complete $n-$vertex graph $K_n$ it was proved that $Kf(K_n)=n-1$ \cite{luk} and since $K_n-v=K_{n-1}$ we get $Kf(K_{n-1})=n-2$.\\
\noindent \textbf{Outcome 3.} $Kf(G)=Kf(G-v)$, i.e. the Kirchhoff index does not change. Here, $v$ is referred to as a
{\em  good vertex} in $G$.  For a fan graph $F_{1,n-1}$, i.e a graph in which a single vertex is connected by an edge to each vertex of a path $P_{n-1}$, the following formula was deduced in \cite{zhang2}:
\begin{equation*}
Kf(F_{1,n-1})=1+n\sum_{k=1}^{n-2}\frac{1}{1+4\sin^2\frac{k\pi}{2(n-1)}}.
\end{equation*}
If we take $n=4$, then $Kf(F_{1,3})=4$. If we remove a vertex $v$ of a degree $3$ from $F_{1,3}$, we  get $Kf(F_{1,3}-v)=Kf(P_3)=W(P_3)=4$. Note that removing a vertex $w$ of degree $2$ gives $Kf(K_{1,3}-w)=Kf(K_3)=2$.\\ 

Furthermore, we are interested in several relaxed versions of
Kirchhoff \v{S}olt\'{e}s problem, all with the primary goal of finding graphs that contain at least one good vertex.  For that purpose, let
\begin{equation*}
K(G)=\{v\in V(G)\,:\, Kf(G)=Kf(G-v)\}
\end{equation*} and let  $0<\beta\leq 1$, $\beta\in\mathbb{Q}$. We say that a graph $G$ is a $\beta-$Kirchhoff \v{S}olt\'{e}s graph if $|K(G)|\geq\beta|V(G)|$. i.e. the proportion of good vertices in $G$ is at least $\beta$. We pose the following problems.
\begin{problem}\label{prob3} For a fixed rational number $\beta \in (0, 1]$ construct an infinite series of $\beta-$Kirchhoff \v{S}olt\'{e}s graphs.
\end{problem}
For simplicity, Kirchhoff \v{S}olt\'{e}s graph is the synonym
for $1-$Kirchhoff \v{S}olt\'{e}s graph. 
A solution to Problem \ref{prob3} for $\beta=1$ would give an infinite series of solutions to Problem \ref{prob2}.\\
\indent Additionally, we aim to construct an infinite family of graphs where the proportion of good vertices increases and approaches a specific real number $\gamma$.
\begin{problem}\label{prob4} Given a fixed $\gamma \in (0, 1]$, construct an infinite series of graphs where the proportion of good vertices increases and asymptotically approaches  $\gamma$ as the order of the graph grows. 
\end{problem}
\indent In this work, we find particular solutions to Problems \ref{prob2} and \ref{prob4} and show that there are infinitely many solutions to Problem \ref{prob3}. Specifically, we demonstrate the existence of at least one Kirchhoff Šoltés graph and present several constructions that produce an infinite series of graphs either with a fixed proportion of good vertices or with a proportion of good vertices approaching a specified value. In Section 2, we introduce various formulas for calculating the resistance distance of a graph and explain the concept of the Kirchhoff index. Moreover, we demonstrate that $C_5$ is the only known Kirchhoff Šoltés graph and that it is the sole example within the class of unicyclic graphs. Section 3 focuses on complete bipartite graphs, presenting a $3/7-$Kirchhoff \v{S}oltés graph. Additionally, an infinite
series of graphs whose proportion of good vertices tends to $3-2\sqrt{2}$  and an infinite series of $1/2-$Kirchhoff Šoltés graphs is constructed.
Finally, in Section 4, we employ a different approach and for each $s\in\mathbb{N}$ we construct an infinite series of graphs whose proportion of good vertices tends to $\frac{s+1}{2s+1}$. This leads to an infinite class of graphs where the proportion of good vertices approaches $2/3$.

\section{Preliminaries}

Let $G$ be a simple connected undirected graph. By $d(v)$  we denote the degree of a vertex $v\in V(G)$. (Sometimes we write $d_G(v)$ if it is important to emphasize the graph $G$.) \\
\indent The Laplacian matrix of a graph $G$ with vertices $v_1,v_2,\ldots,v_n$ is the matrix $L(G)=[l_{ij}]$ where 

$$l_{ij}=\begin{cases} \displaystyle{d(v_i)}\,\,\,\,\,\textnormal{ if } i=j,\\
-1\,\,\,\,\,\,\,\,\,\textnormal{ if } i \textnormal{ and  } j \textnormal{ are adjacent, }\\ 
0\,\,\,\,\,\,\,\,\,\,\,\,\,\,\textnormal{ if } i \textnormal{ and  } j \textnormal{ are non-adjacent. }
\end{cases}$$
\noindent Equivalently, $L(G)=D(G)-A(G)$, where $D(G)=\textnormal{diag}(d(v_1),\ldots, d(v_n))$  is the degree matrix, i.e. the diagonal matrix formed from the vertex degrees of $G$ and  $A(G)$ is the adjacency matrix of $G$. The Laplacian matrix plays an extremely important role in the computation of resistance distance and the Kirchhoff index of a graph. It possesses numerous intriguing properties and has a wide range of applications. Notably, the Laplacian matrix is positive semidefinite and has exactly one zero eigenvalue when $G$
is a connected graph. In this section, we present several important formulas that will be used later.\\
\indent As mentioned in the introduction, resistance distances are computed by methods of the theory of resistive electrical networks. However, in larger and denser graphs it is rather difficult to use series and parallel connection analysis. Instead, the Moore-Penrose generalized inverse of the Laplacian matrix is used \cite{klein1} to calculate the resistance distance $\Omega(v_i,v_j)$ between two vertices $v_i$ and $v_j$ in $G$. We have
\begin{equation}\label{moo}
\Omega(v_i,v_j)=(e_i-e_j)^{\tau}L^{\dagger}(e_i-e_j),
\end{equation}
where $e_i$ denotes the standard unit vector with a $1$ in the $i-$th position and $0$ elsewhere, and $L^{\dag}$ respresents the Moore-Penrose pseudoinverse of the Laplacian matrix.   Formula (\ref{moo}) can be written as 
$$\Omega(v_i,v_j)=(L^{\dagger})_{ii}+(L^{\dagger})_{jj}-(L^{\dagger})_{ij}-(L^{\dagger})_{ji}$$ and since $(L^{\dagger})^{T}=L^{\dagger}$ we get  
$$\Omega(v_i,v_j)=(L^{\dagger})_{ii}+(L^{\dagger})_{jj}-2(L^{\dagger})_{ij}.$$
Gutman and Xiao \cite{gut} used the matrix $\Gamma:=L^{\dag}(G)+ \frac{1}{n}J$, where $J$ denotes the matrix whose all entries  are equal to unity $1$. They obtained
$$\Omega(v_i,v_j)=\Gamma_{ii}+\Gamma_{jj}-2\Gamma_{ij}.$$
Bapat    et al. \cite{bapat} deduced a rather simple formula for the resistance distance:
\begin{equation}\label{det}\Omega(v_i,v_j)=\frac{\textnormal{det} L(i,j)}{\textnormal{det} L(i)}=\frac{\textnormal{det} L(i,j)}{t(G)},
\end{equation}
where $L(i,j)$ is the submatrix obtained from the Laplacain matrix $L$ by deleting its $i-$th and $j-$th rows and $i-$th and $j-$th columns, and $L(i)$ is the submatrix obtained from the Laplacian matrix $L$ by deleting its $i-$th row and $i-$th column. The second equality in (\ref{det}) holds since $\textnormal{det} L(i)=\textnormal{det} L(j)$ for any $i,j\in \{1,\ldots,n\}$ and the famous Kirchhoff matrix tree theorem \cite{kirch} states that  $\textnormal{det} L(i)=t(G)$, where $t(G)$ is the number of spanning trees of $G$. 

Let $\lambda_1,\ldots,\lambda_n$ be the eigenvalues of the Laplacian matrix $L(G)$, known as the Laplacian eigenvalues of $G$. Assuming $\lambda_1\geq \lambda_2\geq\cdots\geq \lambda_n$, we have $\lambda_n=0$. Moreover, $\lambda_{n-1}>0$ because we consider $G$ to be a connected graph \cite{mer}.  Klein \& Randić \cite{klein1} derived a formula for the Kirchhoff index using the trace of the matrix $L^{\dag}$.

\begin{equation*}Kf(G)=n\,\, \textnormal{tr} (L^{\dag}).
\end{equation*} 

Later, Gutman \& Mohar \cite{moh} showed that $\textnormal{tr} (L^{\dag})=\displaystyle{\sum_{i=1}^{n-1}\frac{1}{\lambda_i}}$, thereby establishing a clear relationship between the Laplacian spectrum and the Kirchhoff index:

\begin{equation}\label{mk}Kf(G)=n\sum_{i=1}^{n-1}\frac{1}{\lambda_i}.
\end{equation}
Though computing the Kirchhoff index of graphs is usually challenging, exact values have been successfully calculated for many graphs, and bounds have been derived.  For arbitrary graphs, Lukovits et al. \cite{luk} proved that $Kf(G)\geq n-1$ with equality if and only if $G$ is the complete graph $K_n$. Palacio \cite{palac} showed that $Kf(G)\leq \binom{n+1}{3}$ with equality if and only if $G$ is a path $P_n$. Since Kirchhoff and Wiener index coincide whenever $G$ is acyclic, the case where $G$ is cyclic is especially interesting. Zhang \& Yang\cite{zhang} proved that for any cyclic graph $G$ 
$$n-1\leq  Kf(G)\leq  \frac{n^3-n}{12}.$$
The first equality holds if and only if $G$ is $K_n$ and the second equality holds if and only if $G$ is $C_n$.\\
\indent We will now present an important graph construction known as a splice.
Let $H_1$ and $H_2$ be two non-trivial graphs with $n_1$ and $n_2$ vertices,
respectively, $x_1\in V(H_1)$ and  $x_2\in V(H_2)$. A \textit{splice}
$H_1xH_2$ of graphs $H_1$ and $H_2$ at a vertex $x$ is a graph formed from
$H_1$ and $H_2$ by identifying vertices $x_1$ and $x_2$, so that the resulting
vertex is $x$. While studying the Kirchhoff index of bicyclic graphs, Zhang
 et al. \cite{bicyc}  proved the following:
\begin{equation}\label{identify}
Kf(H_1xH_2)=Kf(H_1)+Kf(H_2)+(n_1-1) Rt_{H_2}(x)+(n_2-1) Rt_{H_1}(x). 
\end{equation}
It is important to note that the equality (\ref{identify}) also applies to the Wiener index and was proven much earlier by Polansky and Bonchev \cite{pol}.
Using (\ref{identify}), we can easily conclude that a pendant vertex, i.e., a vertex of degree $1$ in any arbitrary graph, is not a good vertex. If $H_1$ is a path $P_2$ with end
vertices $v$ and $x$, then $Kf(P_2)=W(P_2)=1$, $Rt_{P_2}(x)=1$ and we get
$Kf(P_2xH_2)-Kf(P_2xH_2-v)=Kf(P_2xH_2)-Kf(H_2)>0$.\\
Notice that vertex $x$ is also not a good vertex since it is a cut vertex, i.e. a vertex whose removal, along with its incident edges, results in a disconnected graph.\\
\indent We close this section  by formalizing
our observation that removing any vertex from the cycle $C_5$ does not affect the Kirchhoff index.

\begin{proposition}\label{c5}
The cycle  $C_5$ is the a Kirchhoff Šolt\'{e}s graph. 
\end{proposition}
\begin{proof}
For $n\geq 3$ the Kirchhoff index of a cycle $C_n$ is given by the formula 
$$Kf(C_n)=\frac{n^3-n}{12}.$$
Removal of an arbitrary vertex $v$ from $C_n$ gives
$$Kf(C_n-v)=Kf(P_{n-1})=W(P_{n-1})=\frac{(n-2)(n-1)n}{6}.$$
Now, the equation $Kf(C_n)=Kf(C_n-v)$ is equivalent to
$n+1=2n-4$ which has a unique solution $n=5$.
\end{proof}

Our conclusions about the pendant vertex, along with Proposition \ref{c5}, imply that $C_5$ is the only unicyclic graph that solves Problem \ref{prob2}.
Our numerical studies failed to identify any additional Kirchhoff Šoltés graphs. We investigated connected graphs with up to $9$ vertices. We employed the geng \cite{geng} software to generate small $n-$vertex $k-$regular graphs. Specifically, we considered the following cases: $k=3$ with $n\leq 4\leq 18$, $k=4$ with $5\leq n\leq 14$, $k=5$ with $6\leq n\leq 12$, $k=6$ with $7\leq n\leq 13$,  $k=7$ with $8\leq n\leq 12$, and $k=8$ with $9\leq n\leq 13$. Additionally, we examined all vertex-transitive graphs with up to 23 vertices using the database compiled by Gordon Royle and Derek Holt, available at \href{https://zenodo.org/records/4010122}{https://zenodo.org/records/4010122} \cite{roy}.   None of these graphs were Kirchhoff Šoltés.

\section{Solutions to relaxed versions of the Kirchhoff Šolt\'{e}s problem}

In this section, we explore relaxed versions of the Kirchhoff-Šoltés problem.  We construct two infinite families of graphs, one of which serves as a particular solution to Problem \ref{prob3}, while the other provides infinitely many solutions to Problem \ref{prob3} and also serves as a specific solution to Problem \ref{prob4}. For simplicity, we focus on graphs whose Laplacian matrices have spectra consisting entirely of integers, known as {\em Laplacian integral graphs}. The simplest examples of such families are the complete graphs and the complete bipartite graphs.

 By $K_n$ we denote the complete graph with $n$ vertices. 
The complete bipartite graph $K_{m,l}$ is a graph whose vertices can be partitioned into independent subsets $V_1$ and $V_2$, $|V_1|=m$, $|V_2|=l$, such that every pair of vertices $\{u,v\}$, $u\in V_1$, $v\in V_2$ is  connected by an edge. The subsets $V_1$ and $V_2$ in a partition are referred to as \textit{blocks}. It is easy to verify that the graphs $K_n$ and $K_{m,l}$ are  Laplacian integral. If we denote the Laplacian eigenvalue $\lambda$ of multiplicity $k$ by $\lambda^{(k)}$, then we can write the spectrum  $\sigma_L$ of the Laplacian matrix of $K_n$ as $\sigma_L(K_n)=\{n^{(n-1)},0\}$, while for $K_{m,l}$ we have
\begin{equation}\label{spec}
\sigma_L(K_{m,l})=\{0,m+l, m^{(l-1)}, l^{(m-1)} \}.
\end{equation}
In the following, we investigate the existence of good vertices in complete bipartite graphs $K_{m,l}$ with $m,l\geq 2.$ We exclude the cases where  $l=1$ or $m=1$  because they result in trees where no vertices are good, as all vertices are either pendant or cut vertices.
\begin{proposition}
Graph $K_{4,3}$ is the only complete bipartite graph with good vertices. Moreover, the vertices of the smaller block are good. Hence, $K_{4,3}$ is a
$3/7$-Kirchhoff \v{S}olt\'es graph.
\end{proposition}
\begin{proof}
By using  (\ref{mk}) we obtain the Kirchhoff index of $K_{m,l}$ as follows 
$$Kf(K_{m,l})=1+(m+l)\frac{m^2+l^2-m-l}{ml}.$$
Notice that $Kf(K_{m,l})$
 is a symmetric function with variables $m$ and $l$. 
Without loss of generality, let $v\in V_2$. Then, the equation $Kf(K_{m,l})=Kf(K_{m,l}-v)$ gives
$$1+(m+l)\frac{m^2+l^2-m-l}{ml}=1+(m+l-1)\frac{m^2+(l-1)^2-m-(l-1)}{m(l-1)},$$
which is equivalent to a binary cubic Diophantine equation
\begin{equation}\label{d1}m[m(m-1)-l(l-1)]-2l(l-1)^2=0.
\end{equation}
As discussed above, we can discard the trivial solution $m = l = 1$. Note that the left-hand side of equation (\ref{d1}) is an increasing function of the variable $m$. It is negative for $m\leq l$ and positive for $m\geq 2l$. Therefore, the solution to the equation (\ref{d1}) is an ordered pair $(m,l)$ such that $l+1\leq m\leq 2l-1$.
It is easy to check that  $(m,l)=(4,3)$ solves (\ref{d1}). Moreover, this is the only solution for which $m=l+1$.\\
To demonstrate that there are no other solutions of the form $m = l + k$ for 
$1 < k < l$, one would need to show that the curve
$$C: \qquad 2y^3 - (2x+4)y^2 - (3x^2 - x - 2)y +x^2 - x^3 = 0,$$
 of genus $0$ (rational curve that can be parametrized by rational functions) obtained by substituting $y = l$ and $x = m-l$ does not contain other points
with nonnegative integer coordinates. This has been verified using the {\em Magma} computer algebra system \cite{bosma} in a private communication with N. Adžaga and G. Dražić, and we omit the details.
  \end{proof}

Before we present further results of this section, it is necessary to define a join of two graphs and the theorems on the Laplacian eigenvalues and the Kirchhoff index of the join of two graphs.\\

The join $H_1+H_2$ of graphs $H_1$ and $H_2$ is a graph with the vertex set $V(H_1+H_2)=V(H_1)\cup V(H_2)$ and the edge set $$E(H_1+H_2)=E(H_1)\cup E(H_2)\cup\{\{v_1,v_2\}\,\,|\,\, v_1\in V(H_1),\,v_2\in V(H_2)\}.$$
Let $|V(H_1)|=n$, $|V(H_2)|=m$, $\sigma_L(H_1)=\{\lambda_1,\ldots,\lambda_n\}$ and $\sigma_L(H_2)=\{\mu_1,\ldots,\mu_m\}$. The following statements hold.
\begin{proposition}\label{prop1}
\cite{mer} The Laplacian eigenvalues of $H_1+H_2$ are 
$$0,\,m+n,\,m+\lambda_i,\,n+\mu_j,\quad i=1,\ldots,n-1;\,j=1,\ldots,m-1.$$
\end{proposition}
\begin{theorem}\label{tzh}\cite{zhang2}
The Kirchhoff index of $H_1+H_2$ is given by
$$Kf(H_1+H_2)=1+(m+n)\left(\sum_{i=1}^{n-1}\frac{1}{m+\lambda_i}+\sum_{j=1}^{m-1}\frac{1}{n+\mu_j}\right).$$
\end{theorem}

Our next result demonstrates that we can construct an infinite class of graphs for which the proportion of good vertices approaches a specific irrational number.

\begin{theorem}\label{five}
There are infinitely many $\beta-$Kirchhoff Šolt\'{e}s graphs, with $\beta$ tending to $3-2\sqrt{2}$. 
\end{theorem}

\begin{proof}

Consider the join  $G=K_s+K_{m,l}$, where $s\geq 1$, $m,l\geq 2$. We study a deletion of an arbitrary vertex of the subgraph $K_s$  of $G$. Since all such vertices belong to the same orbit under the automorphism group of $G$, the goodness of some vertex of the subgraph $K_s$ of $G$ implies the goodness of  all vertices of $K_s$. Therefore, we study the equation $Kf(K_s+K_{m,l})=Kf(K_{s-1}+K_{m,l})$, which is equivalent to equation
\begin{equation}\label{d2}
\frac{l(l-1)}{(s+m)(s+m-1)}+\frac{m(m-1)}{(s+l)(s+l-1)}=1,
\end{equation}
while this equation is equivalent to a quartic Diophantine equation of three variables. Let $k\in\mathbb{N}_0$. We will show that (\ref{d2}) has infinitely many solutions of the form $(s_k,m_k,m_k)$, where the sequences $s_k$ and $m_k$ satisfy the system of linear recurrences with constant coefficients
\begin{equation}\label{system}
\begin{cases}
s_{k+1}=2m_k+s_k-1,  \\[4pt]
m_{k+1}=2s_{k+1}+m_k, 
\end{cases}
\end{equation}
with the initial values $s_0=1$, $m_0=3$.
Note that for $m=n$ equation (\ref{d2}) transforms into 
$2m(m-1)=(s+m)(s+m-1).$
Therefore, we need to prove the equality \begin{equation}\label{d3}
2m_k(m_k-1)=(s_k+m_k)(s_k+m_k-1)
\end{equation} for all $k$.

We use induction on $k$. For $k=0$, i.e. for $(s_0,m_0)=(1,3)$ it is easy to see that  equation (\ref{d3}) holds. 
Let us assume that (\ref{d3}) is true for some integer $k\geq 0$.  We have
    \begin{align*}
        2m_{k+1}(m_{k+1}-1) &=2(2s_{k+1}+m_k)(2s_{k+1}+m_k-1)\\
        &=2(2(2m_k+s_k-1)+m_k)(2(2m_k+s_k-1)+m_k-1)\\
        &=2(5m_k+2s_k-2)(5m_k+2s_k-3)\\
        &=50m_k^2+40m_ks_k-50m_k+8s_k^2-20s_k+12\\
        &=(49m_k^2+42m_ks_k-49m_k+9s_k^2-21s_k+12)\\
        &+(m_k^2-2m_ks_k-m_k-s_k^2+s_k)\\
        &=(7m_k+3s_k-3)(7m_k+3s_k-4)\\
        &+2m_k(m_k-1)-(s_k+m_k)(s_k+m_k-1)
    \end{align*}
    By applying the induction hypothesis, the last two terms of the above-obtained expression vanish and we get 

\begin{align*}
        2m_{k+1}(m_{k+1}-1) &=(7m_k+3s_k-3)(7m_k+3s_k-4)\\
         &=(3(2m_k+s_k-1)+m_k)(3(2m_k+s_k-1)+m_k-1)\\
        &=(3s_{k+1}+m_k)(3s_{k+1}+m_k-1)\\
        &=(s_{k+1}+(2s_{k+1}+m_k))(s_{k+1}+(2s_{k+1}+m_k)-1)\\
        &=(s_{k+1}+m_{k+1})(s_{k+1}+m_{k+1}-1). 
\end{align*}  

To calculate the proportion $\beta_k$ of good vertices in graphs $G_k:=K_{s_k}+K_{m_k,m_k}$, we need to calculate the number $n_k$ of vertices in $G_k$ in terms of $m_k$ and $s_k$. 
By solving  (\ref{system}), we get

\begin{equation}\label{sols}
\begin{cases}
m_k=\displaystyle{\frac{10+7\sqrt{2}}{8}(3+2\sqrt{2})^k+\frac{10-7\sqrt{2}}{8}(3-2\sqrt{2})^k+\frac{1}{2}},  \\[12pt]
s_k=\displaystyle{\frac{3\sqrt{2}+4}{8}(3+2\sqrt{2})^{k}-\frac{3\sqrt{2}-4}{8}(3-2\sqrt{2})^{k}},
\end{cases}
\end{equation}
from which it follows 
\begin{equation}\label{nv}
n_k=2m_k+s_k=\displaystyle{\frac{17\sqrt{2}+24}{8}(3+2\sqrt{2})^{k}-\frac{17\sqrt{2}-24}{8}(3-2\sqrt{2})^{k}+1}.
\end{equation}
\vspace*{5mm}
Now, the proportion of good vertices in $G_k$ is given by $\beta_k=s_k/n_k$. The first few terms of the sequence $(\beta_k)_{k\in\mathbb{N}_0}$ are $$\displaystyle{ \frac{1}{7},\,\frac{1}{6},\,\frac{7}{41},\,\frac{6}{35},\,\frac{41}{239},\,\frac{35}{204}\ldots}$$
and with a simple calculation, we conclude that it is monotonically increasing, i.e. $\beta_{k+1}-\beta_k>0\, \forall k\geq 0$ and 
$$\lim_{k\to \infty}\beta_k=\frac{1}{3+2\sqrt{2}}=3-2\sqrt{2}.$$ 
\end{proof}
Note that $3-2\sqrt{2}<1/5$ which implies that for all $k\geq 0$ it holds $\frac{1}{7}\leq\beta_k<\frac{1}{5}$.\\
\indent Since we aim to get as close as possible to solving the Kirchhoff Šoltés problem by finding families of graphs with the highest possible proportion of good vertices, we present the following theorem.

\begin{theorem}\label{half}
There are infinitely many graphs with half good vertices. 
\end{theorem}
\begin{proof}
    If we let $m=s+l$ and substitute it into equation  (\ref{d2}), we obtain
$$\frac{l(l-1)}{(2s+l)(2s+l-1)}=0 \Leftrightarrow l=1,$$
which implies that $(s,s+1,1)$ is a solution to (\ref{d2}) for every $s\geq 1$. Analogously, $(s,1,s+1)$ is also a solution to (\ref{d2}). Let $G_s:=K_s+K_{s+1,1}$. Then $n_s:=|V(G_s)|=2(s+1)$. Let $v$ be the vertex from the block of the set  $V(K_{s+1,1})$ that contains a single element, i.e. $v\in V_2$. Then $v$ belongs to the same orbit under the automorphism group of $G_s$ as the vertices from $K_s$, making $v$ a good vertex in $G_s$. Therefore, the proportion $\beta$ of good vertices in $G_s$ is equal to $\frac{s+1}{2(s+1)}=\frac{1}{2}$ for any $s\geq 1$. 
\end{proof}

For $s\geq 1$, the graph $G_s$ from Theorem \ref{half} is isomorphic to a graph obtained from $K_{2s+2}$ by deletion of edges of an arbitrary induced subgraph $K_{s+1}$.  Examples of graphs $G_s$ for some values $s$ are shown in Figure \ref{half3}. Note that $G_1\cong F_{1,3}\cong K_4-e$. 

\begin{center}
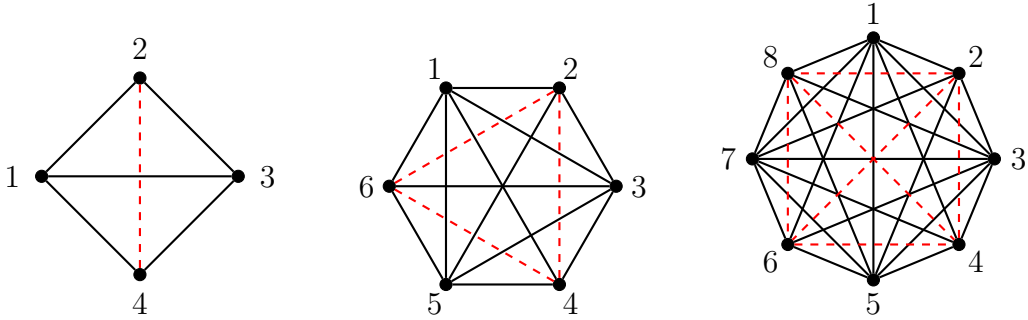
\begin{figure}[htbp]
    \centering
    \begin{minipage}[b]{0.3\textwidth}
        \centering
        \begin{tikzpicture}[scale=1.3]
  \foreach \i/\label in {1/2, 2/1, 3/4, 4/3} {
    \node[vertex] (V\i) at (90*\i:1) {};
    \node at (90*\i:1.3) {\label};
  }
  
  \foreach \i/\j in {1/2, 1/4, 2/3, 2/4, 3/4}
    \draw[edge] (V\i) -- (V\j);
    
  \draw[removed] (V1) -- (V3);
        \end{tikzpicture}
    \end{minipage}
    \hfill
    \begin{minipage}[b]{0.3\textwidth}
        \centering
        \begin{tikzpicture}[scale=1.5]
  \foreach \i/\label in {1/2, 2/1, 3/6, 4/5, 5/4, 6/3} {
    \node[vertex] (V\i) at (360/6*\i:1) {};
    \node at (360/6*\i:1.2) {\label};
  }
  
  \foreach \i/\j in {1/2, 1/4, 1/6, 2/3, 2/4, 2/5, 2/6, 3/4, 3/6, 4/5, 4/6, 5/6}
    \draw[edge] (V\i) -- (V\j);
    
  \draw[removed] (V1) -- (V3);
  \draw[removed] (V3) -- (V5);
  \draw[removed] (V1) -- (V5);
        \end{tikzpicture}
    \end{minipage}
    \hfill
    \begin{minipage}[b]{0.3\textwidth}
        \centering
        \begin{tikzpicture}[scale=0.8]
  \foreach \i/\label in {1/2, 2/1, 3/8, 4/7, 5/6, 6/5, 7/4, 8/3} {
    \node[vertex] (V\i) at (360/8*\i:2) {};
    \node at (360/8*\i:2.4) {\label};
  }
  
  \foreach \i/\j in {1/2, 1/4, 1/6, 1/8, 2/3, 2/4, 2/5, 2/6, 2/7, 2/8, 3/4, 3/6, 3/8, 4/5, 4/6, 4/7, 4/8, 5/6, 5/8, 6/7, 6/8, 7/8}
    \draw[edge] (V\i) -- (V\j);
    
  \draw[removed] (V1) -- (V3);
  \draw[removed] (V3) -- (V5);
  \draw[removed] (V5) -- (V7);
  \draw[removed] (V7) -- (V1);
  \draw[removed] (V1) -- (V5);
  \draw[removed] (V3) -- (V7);
        \end{tikzpicture}
    \end{minipage}
    \caption{Graphs $G_s$ for $s=1,2,3$. The red dashed edges indicate the edges that have been removed from $K_{2s+2}$, i.e. they are not part of $G_s$. Odd-numbered vertices are good vertices.}
    \label{half3}
\end{figure}
\end{center}
\begin{remark} The results in this section are based on solving Diophantine equations. A careful analysis of equation (\ref{d1}) enabled us to establish that
$K_{4,3}$ is the only complete bipartite graph with good vertices. Finding all solutions to the Diophantine equation (\ref{d2}) is a more challenging task. For $m=l=s$, the equation has no solutions. Among particular solutions
$(s,m,l)$, there are infinitely many that satisfy the recurrence (\ref{system}). However, one particular solution, $(s,m,l)=(7,14,8)$, does not fit the recurrence and yields a graph
with $7/29$ good vertices, significantly more than the graphs that satisfy (\ref{system}). Determining all solutions to equation (\ref{d2}) is a challenging task and falls outside the scope of this paper. However, it would be interesting to investigate whether there are other sporadic solutions of (\ref{d2}) yielding higher proportion(s) of good vertices.
\end{remark}

\section{Graph with a proportion of good vertices approaching $2/3$}

In this section, we exploit our infinite family of graphs with half of their vertices being good to construct infinitely many solutions to both Problem \ref{prob3} and Problem \ref{prob4}.  Specifically, for each $s\in \mathbb{N}$, we construct an infinite family of graphs where the proportion of good vertices approaches $\frac{s+1}{2s+1}$. This ultimately leads to a family of graphs where the proportion of good vertices approaches $2/3$, the highest known proportion of good vertices. First, we establish the necessary and sufficient conditions under which a vertex that is good in 
$H_1$ remains good in a splice $H_1xH_2$. 

\begin{proposition}\label{ident}
Let $H_1$ and $H_2$ be two non-trivial graphs with $n_1$ and $n_2$ vertices, respectively, and  $H_1xH_2$ be a splice of $H_1$ and $H_2$ at a vertex $x$. If $v$ is a good vertex in  
$H_1$, $v\neq x$, then $v$ is a good vertex in $H_1xH_2$ if and only if 
\begin{equation}\label{cond}
Rt_{H_2}(x)=(n_2-1)[Rt_{H_1-v}(x)-Rt_{H_1}(x)].
\end{equation}
\end{proposition}
\begin{proof}
From (\ref{identify}) it follows that 
\begin{eqnarray*}
Kf(H_1xH_2)& = & Kf(H_1)+Kf(H_2)+(n_1-1) Rt_{H_2}(x)+(n_2-1) Rt_{H_1}(x)\\
Kf(H_1xH_2-v)& = & Kf(H_1-v)+Kf(H_2)+(n_1-2) Rt_{H_2}(x)\\
&&+(n_2-1) Rt_{H_1-v}(x).
\end{eqnarray*}
 Since $Kf(H_1)=Kf(H_1-v)$, it follows that $Kf(H_1xH_2)-Kf(H_1xH_2-v)=0$ if and only if $Rt_{H_2}(x)=(n_2-1)[Rt_{H_1-v}(x)-Rt_{H_1}(x)]$.
\end{proof}

When constructing a splice, we can take $k\geq 2$ copies $H_i$ of the same graph $H$ and identify the vertices $x_i\in V(H_i)$, $i=1,\ldots,k$, corresponding to the same vertex $x\in V(H)$. We denote the resulting graph by $k\cdot Hx$.
We have the following result. 
 
\begin{proposition}\label{ident_same}
For $m,k\geq 2$, $m,k\in\mathbb{N}$, let $H$ be an arbitrary graph with $m$ vertices that contains a good vertex $v$. Then $v$ is a good vertex in the splice $k\cdot Hx$ of $k$ copies of $H$ at a vertex $x\neq v$ if and only if
\begin{equation}\label{cond2}
mRt_H(x)=(m-1)Rt_{H-v}(x).
\end{equation}
\end{proposition}
\begin{proof}
Let $H_1=H$ and $H_2=(k-1)\cdot Hx$. Then $n_1=m$ and $n_2=(k-1)(m-1)+1$. Since $Kf(H)=Kf(H-v)$, from Proposition \ref{ident} it follows
$Rt_{(k-1)\cdot Hx}(x)=(k-1)(m-1)[Rt_{H-v}(x)-Rt_H(x)].$
Since $Rt_{(k-1)\cdot Hx}(x)=(k-1)Rt_H(x)$, we get 
$$(k-1)Rt_H(x)=(k-1)(m-1)[Rt_{H-v}(x)-Rt_H(x)],$$ which implies (\ref{cond2}).
\end{proof}

\begin{theorem}\label{ksplice}
For each pair $s,k\in\mathbb{N}$, the vertices that are good in $G_s$ remain good in the splice $k\cdot G_s x$ at the vertex $x\in V(G_s)$ such that $d_{G_s}(x)=n_s/2$, where $n_s=2(s+1)=|V(G_s)|$. Moreover, the proportion $\beta_{s,k}$ of good vertices in $k\cdot G_s x$ is given by
\begin{equation}\label{bsk}
\beta_{s,k}=\frac{k(s+1)}{k(2s+1)+1}.
\end{equation}
\end{theorem}

\begin{proof} We need to prove that for any $s\in\mathbb{N}$, the graph $G_s$  satisfies condition (\ref{cond2})  from  Proposition \ref{ident_same}. Let $V(G_s)=\{v_1,\ldots,v_{n_s}\}$, $d(v_1)=d(v_2)=\cdots=d(v_{\frac{n_s}{2}})=\frac{n_s}{2}$ and $d(v_{\frac{n_s}{2}+1})=\cdots=d(v_{n_s})=n_s-1$. Note that the vertices $v_1,\ldots,v_{\frac{n_s}{2}}$ belong to the same orbit under the automorphism group of $G_s$ ($G_s-v_j$, $j=\frac{n_s}{2}+1,\ldots,n_s$), and the same holds for the
vertices $v_{\frac{n_s}{2}+1},\ldots, v_{n_s}$ in $G_s$ ($G_s-v_j$, $j=1,\ldots,n_s/2$).
Therefore, we can take $x=v_1$, $v=v_{n_s}$ and prove that 
\begin{equation}\label{cond2Gs}
n_s Rt_{G_s}(v_1)=(n_s-1)Rt_{G_s-v_{n_s}}(v_1).
\end{equation}
Since in $G_s$ it holds $\Omega(v_1,v_i)=\Omega(v_1,v_j)$ for $i,j\in\{2,\ldots,n_s/2\}$ and $\Omega(v_1,v_i)=\Omega(v_1,v_j)$ for $i,j\in\left\{\frac{n_s}{2}+1,\ldots,n_s\right\}$, we can write
\begin{eqnarray}
Rt_{G_s}(v_1) & = &\left(\frac{n_s}{2}-1\right)\Omega_{G_s}(v_1,v_2)+\frac{n_s}{2}\Omega_{G_s}(v_1,v_{n_s}),\,\,\textnormal{and}\label{rtg1}\\ 
Rt_{G_s-v_{n_s}}(v_1) & =& \left(\frac{n_s}{2}-1\right)\Omega_{G_s-v_{n_s}}(v_1,v_2)\notag\\
&&+\left(\frac{n_s}{2}-1\right)\Omega_{G_s-v_{n_s}}(v_1,v_{n_s-1}).\label{rtg2}
\end{eqnarray}
To calculate the resistance transmissions that appear in (\ref{rtg1}) and (\ref{rtg2}), we use formula (\ref{det}) and calculate determinants of Laplacian submatrices as the products of their eigenvalues.\\
Observe that
\[
L(G_s)=\left[
\begin{array}{r r r r r r r r }
\frac{n_s}{2} & 0 & \cdots & 0 & -1 & -1 & \cdots & -1  \\
0 & \frac{n_s}{2} & \cdots & 0 & -1 & -1 & \cdots & -1  \\
\vdots &  \vdots & \ddots & \vdots  & \vdots & \vdots & \ddots & \vdots   \\
0 & 0 & \cdots & \frac{n_s}{2} & -1 & -1 & \cdots & -1  \\
-1 & -1 & \cdots & -1 & (n_s-1) & -1 & \cdots & -1  \\
-1 & -1 & \cdots & -1 & -1 & (n_s-1) & \cdots & -1  \\
\vdots &  \vdots & \ddots & \vdots  & \vdots & \vdots & \ddots & \vdots   \\
-1 & -1 & \cdots & -1 & -1 & -1 & \cdots & (n_s-1)  \\
\end{array}
\right].
\]
Therefore, it is evident that $\lambda=n_s/2$ and $\lambda=n_s$ are both roots of the polynomial $\textnormal{det} (L(G_s)-\lambda I)$, each with  a multiplicity of at least $(n_s/2)-1$.  Since one of the eigenvalues of $L(G_s)$ is $0$ and  from the fact that $\textnormal{tr}(L(G_s))=\sum_{i=1}^{n_s} \lambda_i$, it is easy to determine the remaining eigenvalue $\lambda$ of $L(G_s)$:
 $$\lambda=\textnormal{tr}(L(G_s))-\left(\frac{n_s}{2}-1\right)\left(n_s+\frac{n_s}{2}\right)=n_s.$$
It follows that $$\sigma_{L}(G_s)=\left\{n_s^{\left(\frac{n_s}{2}\right)},\left(\frac{n_s}{2}\right)^{\left(\frac{n_s}{2}-1\right)},0\right\},$$
and the famous Kirchhoff matrix tree theorem \cite{kirch} gives 
$$t(G_s)=\frac{1}{n_s}\prod_{i=1}^{n_s-1} \lambda_i=\left(\frac{n_s^2}{2}\right)^{\frac{n_s}{2}-1}.$$
Let $s=1$. Then, setting $n_1=4$, straightforward calculations show that $\Omega_{G_1}(v_1,v_2)=1$, $\Omega_{G_s}(v_1,v_4)=5/8$, $\Omega_{G_1-v_4}(v_1,v_2)=2$ and $\Omega_{G_1-v_4}(v_1,v_3)=1$. Furthermore,  $Rt_{K_4-e}(1)=9/4$ and $Rt_{K_4-e-2}(1)=3$ which implies that $G_1$ satisfies  (\ref{cond2Gs}).\\
Let $s\geq 2$.  To calculate $\Omega_{G_s}(v_1,v_2)$, we need to determine the eigenvalues of the submatrix $L(G_s)(v_1,v_2)$ of the Laplacian matrix $L(G_s)$, which is obtained by deleting the rows and columns corresponding to vertices $v_1$ and $v_2$. From the structure of the matrix $L(G_s)(v_1,v_2)$, it is clear that  $n_s/2$ is an eigenvalue with multiplicity at least $\frac{n_s}{2}-3$ and $n_s$ is an eigenvalue with multiplicity at least $\frac{n_s}{2}-1$. The remaining two eigenvalues of $L(G_s)(v_1,v_2)$ are  $$\displaystyle{\mu_{1,2}=\frac{n_s}{2}\pm\sqrt{\frac{n_s(n_s-4)}{4}}}$$  with the associated eigenvectors $$\vec{x}_{1,2}=[\underbrace{\mp\sqrt{\frac{n_s}{n_s-4}},\ldots,\mp\sqrt{\frac{n_s}{n_s-4}}}_{(n_s/2)-2 \textnormal{ terms }},\underbrace{1,\ldots,1}_{n_s/2 \textnormal{ terms }}]^t.$$
 This can be verified directly by substituting obtained expressions into the equations $L(G_s)(v_1,v_2)\vec{x}_i=\mu_i \vec{x}_i$, $\vec{x}_i\in\mathbb{R}^{n_s-2}$, $\vec{x}_i\neq \vec{0}$, $i=1,2$. \\
 
Therefore, $\sigma(L(G_s)(v_1,v_2))=\left\{n_s^{\left(\frac{n_s}{2}-1\right)},\left(\frac{n_s}{2}\right)^{\left(\frac{n_s}{2}-3\right)},\frac{n_s}{2}\pm \sqrt{\frac{n_s(n_s-4)}{4}}\right\}$, from which it follows that
$\displaystyle{\textnormal{det}L(G_s)(v_1,v_2)=\prod_{\mu\in\sigma(L(G_s)(v_1,v_2))}}\mu=n_s^{\frac{n_s}{2}}\left(\frac{n_s}{2}\right)^{\frac{n_s}{2}-3}$. We obtain
\begin{equation}\label{12}
\Omega_{G_s}(v_1,v_2)=\frac{\textnormal{det}L(G_s)(v_1,v_2)}{t(G_s)}=\frac{4}{n_s}.
\end{equation}
By using similar arguments, we can calculate $\Omega_{G_s}(v_1,v_{n_s})$. The structure of $L(G_s)(v_1,v_{n_s})$ implies that $\frac{n_s}{2}$ and $n_s$  are among its eigenvalues, each with multiplicity at least $\frac{n_s}{2}-2$. 
It is easy to verify that the remaining two eigenvalues of $L(G_s)(v_1,v_{n_s})$ are  $$\displaystyle{\mu_{1,2}=\frac{1}{2}\left(n_s+1\pm\sqrt{n_s^2-4n_s+5}\right)},$$  with the associated eigenvectors $$\vec{x}_{1,2}=[\underbrace{-\frac{n_s-2}{1\pm\sqrt{n_s^2-4n_s+5}},\ldots,-\frac{n_s-2}{1\pm\sqrt{n_s^2-4n_s+5}}}_{(n_s/2)-1 \textnormal{ terms }},\underbrace{1,\ldots,1}_{(n_s/2)-1 \textnormal{ terms }}]^t.$$
We get $\sigma(L(G_s)(v_1,v_{n_s}))=\left\{n_s^{\left(\frac{n_s}{2}-2\right)},\left(\frac{n_s}{2}\right)^{\left(\frac{n_s}{2}-2\right)},\frac{1}{2}(n_s+1\pm\sqrt{n_s^2-4n_s+5})\right\}$, from which it follows that
$$\displaystyle{\textnormal{det}L(G_s)(v_1,v_{n_s})=\prod_{\mu\in\sigma(L(G_s)(v_1,v_{n_s}))}}\mu=\frac{3n_s-2}{2}\left(\frac{n_s^2}{2}\right)^{\frac{n_s}{2}-2}.$$
Therefore,
\begin{equation}\label{1n}
\Omega_{G_s}(v_1,v_{n_s})=\frac{\textnormal{det}L(G_s)(v_1,v_{n_s})}{t(G_s)}=\frac{3n_s-2}{n_s^2}.
\end{equation}    
By inserting (\ref{12}) and (\ref{1n}) into (\ref{rtg1}), we obtain
\begin{equation}\label{rtg11}
Rt_{G_s}(v_1)=\frac{7}{2}-\frac{5}{n_s}.
\end{equation}
To calculate $\Omega_{G_s-v_{n_s}}(v_1,v_2)$ and $\Omega_{G_s-v_{n_s}}(v_1,v_{n_s-1})$, we  first need to determine the number $t(G_s-v_{n_s})$ of spanning trees in $G_s-v_{n_s}$. Without delving into the details, from the spectrum of $L(G_s-v_{n_s})$, which is easy to establish, we get 
$$t(G_s-v_{n_s})=(n_s-1)^{\frac{n_s}{2}-2}\left(\frac{n_s}{2}-1\right)^{\frac{n_s}{2}-1}.$$
Furthermore, the spectrum of the matrix $L(G_s-v_{n_s})(v_1,v_2)$ is
$$\sigma (L(G_s-v_{n_s})(v_1,v_2))=\left\{(n_s-1)^{\left(\frac{n_s}{2}-2\right)},\left(\frac{n_s}{2}-1\right)^{\left(\frac{n_s}{2}-3\right)},n_s-2,1\right\}.$$
The eigenvalues $n_s-1$ and $\frac{n_s}{2}-1$ with their multiplicities are easy to determine from the structure of the matrix $L(G_s-v_{n_s})(v_1,v_2)$.  The eigenvalues $n_s-2$ and $1$ correspond to eigenvectors $$[\underbrace{\frac{n_s}{2}-1,\ldots,\frac{n_s}{2}-1}_{(n_s/2)-2 \textnormal{ terms }},\underbrace{\frac{n_s}{2}-2,\ldots,\frac{n_s}{2}-2}_{(n_s/2)-1 \textnormal{ terms }}]^t$$ and 
$$[\underbrace{-1,\ldots,-1}_{(n_s/2)-2 \textnormal{ terms }},\underbrace{1,\ldots,1}_{(n_s/2)-1 \textnormal{ terms }}]^t,$$ 
respectively. 
It follows that 
\begin{equation}\label{12-v}
\Omega_{G_s-v_{n_s}}(v_1,v_2)=\frac{\textnormal{det}L(G_s-v_{n_s})(v_1,v_2)}{t(G_s-v_{n_s})}=\frac{4}{n_s-2}.
\end{equation}
Finally, the spectrum of the matrix $L(G_s-v_{n_s})(v_1,v_{n_s-1})$ is
$$\left\{(n_s-1)^{\left(\frac{n_s}{2}-3\right)},\left(\frac{n_s}{2}-1\right)^{\left(\frac{n_s}{2}-2\right)},\frac{1}{2}(n_s\pm \sqrt{n_s^2-6n_s+12})\right\}.$$
Note that the eigenvalues $\frac{1}{2}(n_s\pm \sqrt{n_s^2-6n_s+12})$  correspond to eigenvectors $$[\underbrace{-\frac{n_s-4}{2\pm\sqrt{n_s^2-6n_s+12}},\ldots,-\frac{n_s-4}{2\pm\sqrt{n_s^2-6n_s+12}}}_{(n_s/2)-1 \textnormal{ terms }},\underbrace{1,\ldots,1}_{(n_s/2)-2 \textnormal{ terms }}]^t.$$
We obtain
\begin{equation}\label{1n-v}
\Omega_{G_s-v_{n_s}}(v_1,v_{n_s-1})=\frac{\textnormal{det}L(G_s-v_{n_s})(v_1,v_{n_s-1})}{t(G_s-v_{n_s})}=\frac{3}{n_s-1}.
\end{equation}
By inserting (\ref{12-v}) and (\ref{1n-v}) into (\ref{rtg2}), we obtain
\begin{equation}\label{rtg22}
Rt_{G_s-v_{n_s}}(v_1)=\frac{7n_s-10}{2(n_s-1)}.
\end{equation}
We conclude that (\ref{rtg11}) and (\ref{rtg22}) satisfy (\ref{cond2Gs}).
Since the number of vertices in $k\cdot G_s x$ is equal to $k(2s+1)+1$ and the number of good vertices is $k(s+1)$, we obtain $\beta_{s,k}$ as in (\ref{bsk}).
\end{proof}

It is easy to verify that the sequence of proportions $(\beta_{s,k})_{s,k\in\mathbb{N}}$  of good vertices from Theorem \ref{ksplice} decreases in variable $s$ and increases in variable $k$. For a fixed $s$, it tends to $\frac{s+1}{2s+1}$. Therefore, by taking $s=1$, we get the highest possible limiting value of the proportion of good vertices: 
$$\lim_{k\to\infty} \beta_{1,k}=\lim_{k\to\infty}\frac{2k}{3k+1}=\frac{2}{3}.$$
A graph with $\beta_{1,4}=8/13$ good vertices is shown in Figure \ref{fig4}.

\begin{corollary}\label{2_3}
For each $s\in\mathbb{N}$ there exists an infinite family of graphs where the proportion of good vertices approaches $\frac{s+1}{2s+1}.$
For $s=1$ this proportion approaches $2/3$.\hfill\qed
\end{corollary}

\begin{center}
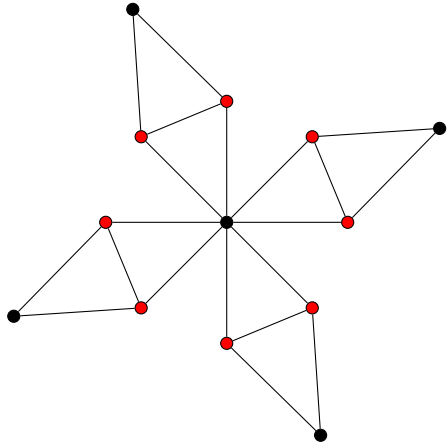
\begin{figure}[htbp]
\centering
\begin{tikzpicture}[scale=0.8]
    \node[draw, circle, fill=black, minimum size=1.6mm, inner sep=0pt] (1) at (0,0) {};

    \foreach \i [count=\j from 2] in {90, 135, 180, 225, 270, 315, 0, 45} {
        \node[draw, circle, fill=red, minimum size=1.6mm, inner sep=0pt] (\j) at (\i:2) {};
    }

    \foreach \j in {2,3,4,5,6,7,8,9} {
        \draw (1) -- (\j);
    }

    \draw (2) -- (3);
    \draw (4) -- (5);
    \draw (6) -- (7);
    \draw (8) -- (9);

    \node[draw, circle, fill=black, minimum size=1.6mm, inner sep=0pt] (10) at ($ (2)!.5!(3) + (115:2) $) {};
    \node[draw, circle, fill=black, minimum size=1.6mm, inner sep=0pt] (11) at ($ (6)!.5!(7) + (295:2) $) {};
    \node[draw, circle, fill=black, minimum size=1.6mm, inner sep=0pt] (12) at ($ (8)!.5!(9) + (25:2) $) {};
    \node[draw, circle, fill=black, minimum size=1.6mm, inner sep=0pt] (13) at ($ (4)!.5!(5) + (205:2) $) {};

    \draw (10) -- (2);
    \draw (10) -- (3);
    \draw (11) -- (6);
    \draw (11) -- (7);
    \draw (12) -- (8);
    \draw (12) -- (9);
    \draw (13) -- (4);
    \draw (13) -- (5);

\end{tikzpicture}
\caption{A splice $4\cdot G_1 x$ identified at a vertex $x$ of  degree $2$. The vertices of degree $3$ are good and are colored in red.}
\label{fig4}
\end{figure}
\end{center}

\section{Concluding remarks}

In this paper, we addressed the problem of identifying graphs whose Kirchhoff index, defined as the sum of resistance distances between all pairs of vertices, remains unchanged after removing an arbitrary vertex.  This problem is a resistance distance variant of the (in)famous \v{S}olt\'{e}s
problem for the Wiener index, i.e., for the shortest-path distance.  Similar to the original problem, where the only known solution is the cycle  $C_{11}$ on
eleven vertices, the only solution identified for the resistance distance problem is also a cycle, specifically  $C_5$. The resemblance to the original problem is further highlighted by the fact that the existence of other solutions has not been ruled out and the highest known proportion of good vertices is not more than $\frac{2}{3}$. However, we demonstrated that any potential additional solutions cannot belong to certain graph classes, such as graphs with cut vertices, complete graphs, and complete bipartite graphs.

We also explored several relaxed versions of the Kirchhoff Šolt\'{e}s problem, where the common goal is to find graphs with at least one good vertex. We studied the $\beta-$Kirchhoff Solt\'{e}s problem, aiming to find an infinite family of graphs where the proportion
of good vertices is at least $\beta$. Furthermore, we investigated the problem of constructing infinite families
of graphs for which the proportion of good vertices grows and asymptotically approaches a certain real
number as the order of a graph increases. We found infinitely many solutions for both relaxed versions.
Despite the similarities (sums of pairwise distances, only one known solution being a cycle, and partial solutions to relaxed versions, some of which use (\ref{identify}), which applies to both the Wiener and Kirchhoff indices), the two versions of the \v{S}olt\'{e}s problem require different approaches, primarily due to the Kirchhoff index's definition via the Laplacian spectrum. In our studies, we relied on Laplacian integral graphs and structural results from the theory of electric networks to provide partial solutions to the relaxed versions of the problem.

Several intriguing questions remain unanswered in this paper. It would be interesting to investigate whether a more comprehensive analysis of equation (\ref{d2}) might yield sporadic solutions with larger fractions of good vertices. Another promising direction could involve examining other classes of Laplacian integral graphs \cite{kirkland} and assessing their potential for constructing infinite families of graphs with high fraction(s) of good vertices. We believe that both approaches used here, involving Diophantine equations and the splicing method, could be effectively applied to such graph classes.

We conclude this section by proposing the following problem.
\begin{problem} \label{last}
Which (rational or irrational) numbers $\gamma \in (0,1]$ appear either as exact or as 
limiting values of proportions of good vertices with respect to the
Kirchhoff \v{S}olt\'es problem for simple undirected graphs? 
\end{problem}

\vskip 1pc
\section*{Acknowledgements}
This work was supported by the Croatian Science Foundation under project number HRZZ-IP-2024-05-2130. Partial support of the Slovenian ARRS (program P1-0383, grant no. J1-3002) is
gratefully acknowledged by T. Do\v{s}li\'c. 
The authors are thankful to Nikola Ad\v{z}aga and Goran Dra\v{z}i\'c for their
help with analysis of equation (\ref{d1}).


\bibliographystyle{model1-num-names}

\end{document}